\documentclass{amsart}

\usepackage[english]{babel}
\usepackage[utf8x]{inputenc}
\usepackage[T1]{fontenc}

\usepackage[top=3cm,bottom=2cm,left=3cm,right=3cm,marginparwidth=1.75cm]{geometry}

\usepackage{pdfpages}

\usepackage{amsmath,amsthm,amsfonts,amssymb,times,latexsym,enumerate,comment,marginnote,slashed,dsfont}
\usepackage{pdfpages, relsize, setspace}
\usepackage{xypic}
\usepackage[OT2,T1]{fontenc} 

\DeclareSymbolFont{cyrletters}{OT2}{wncyr}{m}{n}
\DeclareMathSymbol{\sha}{\mathalpha}{cyrletters}{"58}

\newcommand{\abs}[1]{\left\lvert#1\right\rvert}
\newcommand{\norm}[1]{\lVert#1\rVert}

\newcommand{\mbf}{\mathbb{F}}
\newcommand{\mbr}{\mathbb{R}}
\newcommand{\mbz}{\mathbb{Z}}

\newcommand{\mbc}{\mathbb{C}}
\newcommand{\mbq}{\mathbb{Q}}

\renewcommand{\:}{\colon}

\newcommand{\ra}{\rightarrow}

\newcommand{\bs}{\backslash}
\newcommand{\ssq}{\subseteq}

\newcommand{\wtilde}{\widetilde}

\DeclareMathOperator{\ord}{ord}

\newcommand*{\longhookrightarrow}{\ensuremath{\lhook\joinrel\relbar\joinrel\rightarrow}}
\newcommand*{\inj}{\longhookrightarrow}

\newcommand{\bbm}{\begin{bmatrix}}
\newcommand{\ebm}{\end{bmatrix}}
\newcommand{\bpm}{\begin{pmatrix}}
\newcommand{\epm}{\end{pmatrix}}
\newcommand{\xym}[1]{\xymatrix{#1}}

\newcommand{\gen}[1]{\langle #1 \rangle}

\DeclareMathOperator{\im}{Im}
\renewcommand{\Im}{\im}

\newtheorem{theorem}{Theorem}[section]
\newtheorem{lemma}[theorem]{Lemma}
\newtheorem{proposition}[theorem]{Proposition}
\newtheorem*{theorem*}{Theorem}
\newtheorem{corollary}[theorem]{Corollary}

\newtheorem{question}[theorem]{Question}
\newtheorem{example}[theorem]{Example}

\makeatletter
\newtheorem*{rep@theorem}{\rep@title}
\newcommand{\newreptheorem}[2]{%
	\newenvironment{rep#1}[1]{%
		\def\rep@title{#2 \ref{##1}}%
		\begin{rep@theorem}}%
		{\end{rep@theorem}}}
\makeatother

\newreptheorem{theorem}{Theorem}
\newreptheorem{conjecture}{Conjecture}

\theoremstyle{definition}
\newtheorem{definition}[theorem]{Definition}

\newtheorem{remark}[theorem]{Remark}

\usepackage[alphabetic]{amsrefs}
\usepackage{tabu}
\usepackage[noend]{algorithmic} 

\newtheoremstyle{algorithmstyle}
{10pt}      
{5pt}       
{}  
{}          
{\bfseries} 
{}         
{ }      
{}          

\theoremstyle{definition}
\newtheorem{problem}{Problem}
\newtheorem*{assumptions*}{Assumptions}
 
\theoremstyle{algorithmstyle}
\newtheorem{algorithm}[theorem]{Algorithm}

\title{Solving p-adic polynomial systems via iterative eigenvector algorithms}
\author{Avinash Kulkarni}
\address{Max Planck Institute MIS Leipzig}
\email{avinash@mis.mpg.de}

\date{\today}

\subjclass[2010]{15A18 (primary), 11S05 (secondary)}
\keywords{p-adic linear algebra, solving polynomial systems, eigenvector algorithms.}

\DeclareMathOperator{\GL}{GL}
\DeclareMathOperator{\Ogroup}{O}

\begin{document}

\maketitle

\begin{abstract}
	In this article, we describe an implementation of a polynomial system solver to compute the approximate solutions of a $0$-dimensional polynomial system with finite precision $p$-adic arithmetic. We also describe an improvement to an algorithm of Caruso, Roe, and Vaccon for calculating the eigenvalues and eigenvectors of a $p$-adic matrix.
\end{abstract}

\section{Introduction}

	Let $k$ be a field and let $f_1, \ldots, f_m \in k[x_1, \ldots, x_n]$ be polynomials such that the variety defined by $f_1, \ldots, f_m$ has dimension $0$. It is often of interest to compute the solutions of such a system, either exactly or approximately. A significant bottleneck in computing the exact solutions of such a polynomial system is the complexity of the field extension required to write down all of the solutions. A popular method to compute the solutions exactly is to compute a triangular decomposition for the ideal $\gen{f_1, \ldots, f_m}$ and solve for the coordinates via back-substitution. A difficulty with this approach is the explosion of the size of the coefficients, which in conjunction with taking complicated field extensions, usually renders a computer algebra system unresponsive.
	
	Computations in an exact field (such as $\mbq$) can often be approximated by computations in an inexact field (such as $\mbr, \mbc$), where by
	inexact, we mean that the computer representation of an element is only an approximation up to some precision. It is often far faster to compute an approximation to the solutions, though the trade-off is that the solutions are not known perfectly. In numerous applications, a real or complex approximation to the solutions is sufficient.
	
	In Linear Algebra, $p$-adic methods go back to at least Dixon~\cite{dixon1982exact}. One advantage pointed out by Dixon is that a linear system can be ill-conditioned with respect to an archimedean norm, but well-conditioned $p$-adically; general polynomial systems exhibit this behavior as well. Recently, new developments in number theory and tropical geometry highlight that non-archimedean data is not just a route to integral solutions, but is significant in and of itself.
	
	Our polynomial system solver is based on the truncated normal form solver of \cite{TMV2018truncated}. To transplant their implementation to the $p$-adic setting, we rely on finite precision linear algebra over the $p$-adic field, which in analogy to the case over $\mbr$ we refer to as \emph{pnumerical linear algebra}. It is essential for us to be able to compute the eigenvalues and eigenvectors of an approximate $p$-adic matrix quickly and stably. The naive algorithm to compute the eigenvalues of a matrix relies on the calculation and factorization of the characteristic polynomial. In the numerical setting (i.e, over $\mbr$) it is desirable to avoid this step due to the instability of solving for the roots of a polynomial. There is a similar loss of precision in solving for the roots of a $p$-adic polynomial when the roots of the polynomial are $p$-adically close together. Unlike the archimedean setting, the accurate calculation of the characteristic polynomial is more difficult since a division free algorithm must be used \cite[Introduction]{CRV2017characteristic}. Our idea to improve on the algorithm of \cite{CRV2017characteristic} is to adapt the ideas from classical numerical linear algebra to the pnumerical setting; specifically to use an iterative scheme to compute the eigenvectors or eigenvalues. To our knowledge, this is the first appearance of a $p$-adic numerical algorithm based on iterating matrix multiplication to solve for eigenvalues and eigenvectors.
	
	The significant deviation from the classical numerical setting is the preconditioning strategy. A general matrix over $\mbr$ tends to have eigenvalues of distinct absolute values, so aspects of the iteration are dominated by the unique largest eigenvalue. In the $p$-adic setting, it tends to be the case that all of the eigenvalues have identical absolute value, and this is true for most shifts of the form $A \mapsto A-\mu I$ as well. Our strategy is to first approximate the eigenvalues by computing the characteristic polynomial over the residue field, then refine the initial approximation via iteration. The advantage of this approach is that the computation of the characteristic polynomial of a matrix in $M_n(\mbf_q)$ is significantly faster than computing the characteristic polynomial over $\mbq_p$~\cite{CRV2017characteristic,Storjohann2001frobenius}. The basic Hensel-lifting based strategy applied to the system of equations $Av = \lambda v, {v_i = 1}$, with $v_i$ some fixed entry of $v$, would require the computation of the inverse of the Jacobian matrix of the map $F\: (v,\lambda) \mapsto (Av - \lambda v, v_i-1)$ at each step in the iteration. Our cost per iteration step is much smaller. Unfortunately, matrices of the form $\mu I + N$, with $N \in M_n(\mbz_p)$ topologically nilpotent, do not allow us to apply our deflation strategy as the characteristic polynomial over the residue field is unhelpful; we hope that future improvements to our strategy can be made.
	
	The structure of the article is as follows. In Section~\ref{sec: linear algebra} we discuss some background and terminology regarding the subject of finite precision computation over $\mbq_p$. We review some specific results on linear algebra over the $p$-adic field, especially the $p$-adic analogues of condition-stable algorithms based on matrix factorizations. We describe our iterative strategy for the computation of the eigenvectors of a matrix over $\mbq_p$ and compare it to the existing alternatives. In Section~\ref{sec: solve equations}, we describe our adaptations of the algorithm of \cite{TMV2018truncated} to solve polynomial systems over $\mbq_p$. Our algorithms are available as \texttt{Julia} packages, and can be obtained from: 
	\begin{center}
		\begin{tabular}{l}
			\url{https://github.com/a-kulkarn/pAdicSolver} \\
			\url{https://github.com/a-kulkarn/Dory}
		\end{tabular}
	\end{center}
	(For technical reasons related to the \texttt{Julia} package system, we choose to divide our implementation between two packages.)

\section{Linear algebra} \label{sec: linear algebra}

	\subsection{Precision, norms, and condition numbers}
	
		We state some basic definitions for our discourse. We direct the reader to \cites{CRV2015linear, caruso2017computationswithpadics, Kedlaya2010differential} for more details. For a matrix $A \in M_n(\mbz_p)$, we will denote by $\chi_A$ the characteristic polynomial and denote $\chi_{A,p} := \chi_{A} \pmod p$. We denote the standard basis of $\mbq_p^n$ (or $\mbz_p^n$) by $\{ e_1, \ldots, e_n\}$.

		\begin{definition}
			Let $v := (v_1, \ldots, v_n) \in \mbq_p^n$ be a vector. The \emph{norm} of $v$ is
				\[
					\norm{v}_p := \sup_{1 \leq i \leq n} \{ \abs{v_i}_p \}.
				\]
		\end{definition}
	
		\begin{definition}
			Let $A \in \mbq_p^{n \times m}$ be a matrix. The \emph{operator norm} of $A$ with respect to $\norm{\cdot}_p$ is
				\[
					\norm{A}_p := \sup_{v \in \mbq_p^m \bs \{0\}} \frac{\norm{Av}_p}{\norm{v}_p}.
				\]
			The \emph{condition number} for an invertible square matrix is $\kappa(A) := \norm{A}_p \cdot \norm{A^{-1}}_p$. The condition number for a singular matrix is $\infty$.
		\end{definition}
	
		We can identify a subgroup of $\GL_n(\mbq_p)$ where every matrix is well-conditioned, serving the analogous role to $\operatorname{O}_n(\mbr)$ in the real setting.
			
		\begin{lemma}
			Let $A \in \GL_n(\mbz_p)$. Then $\kappa(A) = 1$.
		\end{lemma}
	
		\begin{proof}
			See the comments following \cite[Definition 4.3.3]{Kedlaya2010differential}.
		\end{proof}
	
		We now come to the discussion of $p$-adic precision. There are many possible ways to represent a $p$-adic element  $a \in \mbq_p$  in a computer system \cite{caruso2017computationswithpadics}. In the \texttt{Nemo/Hecke} systems \cite{nemohecke}, an element is represented by a series
			\[
				a = a_{-r}p^{-r} + \ldots + a_0 + pa_1 + p^2 a_2 + \ldots + a_{N-1}p^{N-1} + O(p^N)
			\]
		where the $O(p^N)$ is the $p$-adic ball representing the uncertainty of the remaining digits. The \emph{relative precision} of $a$ is the quantity $N + r$, and the \emph{absolute precision} is the number $N$. In the terminology of \cite{caruso2017computationswithpadics}, we consider a system with the \emph{zealous} (i.e, \emph{interval}) implementation of arithmetic. The operations $-, +$ preserve the minimum of the absolute precision of the operands, and $\times, \div$ preserve the minimum relative precision of the operands. If $u \in \mbz_p^\times$, $a \in \mbz_p$, and $N \leq N'$, then we have that $(u+O(p^{N'}))(a + O(p^{N})) = ua + O(p^N)$. Multiplication by $p$ preserves the relative precision and increases the absolute precision by $1$. The worst operation when it comes to absolute $p$-adic precision is dividing a small number by $p$. For example, the expression
			\[
			\frac{(1 + p^{99} + O(p^{100})) - (1 + O(p^{100}))}{p^{100} + O(p^{200})} = p^{-1} + O(1)
			\]
		begins with $3$ numbers with an absolute and relative precision of at least $100$, and ends with a result where not even the constant term is known. 
				
		In \texttt{Nemo/Hecke}, elements of a matrix store their own precision. In the terminology of \cite{CRV2017characteristic}, the associated precision structure is the jagged precision lattice. However, we will not keep track of this finer precision here. Instead, in the terminology of \cite{CRV2017characteristic}, we will look at the flat precision of the matrix:		
		\begin{definition} \label{def: matrix oh-notation}
			Let $A,B \in M_n(\mbz_p)$ be matrices such that $a_{i,j} = b_{i,j} + O(p^{N_{i,j}})$. Then we write $A = B + O(p^N)$, where $N := \min_{i,j} N_{i,j}$.
		\end{definition}
		To refer to a matrix $A \in M_n(\mbq_p)$ whose elements are known at an absolute precision at least $N$, we will simply write $A + O(p^N)$. 
		
		\subsubsection{Design choices for the implementation}
		
		We discuss some of the design choices of the computer algebra package accompanying this article. The user we had in mind while designing our system is one who is interested in data that is inherently $p$-adic. More specifically, we have assumed the following hypotheses: 
		
		\begin{assumptions*}
			\
			\begin{itemize}
			\item
				Input and output are inherently approximate.
			\item
				The desired output precision is roughly the input precision.
			\item
				Element-wise gains in absolute precision after $a \mapsto pa$ or $a \mapsto a^p$ are usually lost in a subsequent step.
			\end{itemize}
		\end{assumptions*}
		
		We believe that our linear algebra package will be used in the middle of an ongoing $p$-adic computation, such as solving polynomial systems of equations. In such a circumstance, it does not make sense for our algorithm to ``request'' extra digits of precision from the input data since the input itself is approximate. We also assume that if a user inputs data at precision $N$, they expect output at precision roughly $N$, or at least $N$ minus the condition number for the problem. Computations using relaxed $p$-adic numbers \cite{caruso2017computationswithpadics} allow the desired output precision to be specified. Unfortunately, relaxed arithmetic does not seem well-suited for use in an iterative scheme due to the extra memory costs (see \cite[Section 2.4]{caruso2017computationswithpadics}). We also assume for a dense matrix that the entries are all given at the same flat absolute precision. We believe that this situation is fairly common in practical applications, but of course not ubiquitous. That said, the implementation of $p$-adic matrices in \cite{nemohecke} uses an element-wise precision model, so it is possible for a computation using our software to validly return a result with more accuracy than we indicate here.
				
		In our analysis, we do not account for the extra precision on an element $a \in \mbq_p$ gained after $a \mapsto pa$ or $a \mapsto a^p$. Experimentally, this seems to be a reasonable assumption for randomly selected dense matrices over a $p$-adic field, with $p \geq 40$. For a more structured scenario this assumption might not be reasonable. A nice feature of capped precision is that we do not have to account for the growth of arithmetic costs due to increased precision. 
		
		With these assumptions made, it is often useful to think of the computer as performing arithmetic in a $\mbz/p^N\mbz$-module for a computation with no divisions by $p$. It is useful to consider the image of a matrix $A \in M_n(\mbz_p)$ in the ring $M_n(\mbz/p^N\mbz)$. It is an essential difficulty of finite precision $p$-adic linear algebra that $\mbz/p^N\mbz$ is a local principal ideal ring, but not a domain when $N > 1$.
		
		\begin{remark}
			Our eigenvector algorithms are focused on finding only the eigenvectors defined over $\mbq_p$. We made this choice for the following two reasons. First, in the last section of the article, we focus on finding the $\mbq_p$-solutions of a $0$-dimensional system of polynomial equations over $\mbq_p$. Secondly, our present implementation is written using components of the OSCAR system (specifically \cite{nemohecke}) that are currently under development. At the time of writing the first version of our implementation, the ``$q$-adics'' interface or an interface to handle ramifed extensions of $\mbq_p$ did not exist. Thus, our article is written with $\mbq_p$ in mind to more closely reflect the actual software. Rapid progress is being made on the system, so we expect our implementation to evolve over time as well.
		\end{remark}
			
	\subsection{Matrix factorizations}

	In this section, we very briefly review some matrix invariants and factorizations. Matrix factorizations relating to the topological structure of $\mbq_p$ are unsurprisingly related to the algebraic structure of the $\mbz_p$-module spanned by the columns (or rows) of the matrix. We invite the reader to consider \cite[Chapter 4]{Kedlaya2010differential} for further details.

	\begin{proposition}[Iwasawa decomposition] \label{prop: iwasawa decomposition}
		Let $k$ be either a finite extension of $\mbq_p$, $\mbr$, or $\mbc$ and let $G \inj \GL_n(k)$ be a linear algebraic group, let $K$ be a maximal compact subgroup of $G$, and let $B$ be a Borel subgroup. For any $A \in G$, there exists a $Q \in K$ and an $R \in B$ such that $A = QR$. 
	\end{proposition}

	Note that a maximal compact subgroup of $\GL_n(\mbr)$ is the orthogonal group $\operatorname{O}_n(\mbr)$, and the subgroup of invertible real upper triangular matrices is a Borel subgroup. We have as a consequence:

	\begin{corollary}[$QR$-factorization]
		Let $A \in \mbr^{n \times m}$ be a matrix. Then there exists a $Q \in \operatorname{O}_n(\mbr)$ and an upper triangular matrix $R \in \mbr^{n \times m}$ such that $A = QR$. 
	\end{corollary}
	
	Of course, $QR$-factorization is really just a consequence of Gram-Schmidt. From the point of view of Proposition~\ref{prop: iwasawa decomposition}, we also have a $p$-adic analogue of $QR$-factorization. In $\GL_n(\mbq_p)$, a maximal compact subgroup is $\GL_n(\mbz_p)$, and the subgroup of invertible upper triangular matrices is a Borel subgroup.
	
	\begin{corollary}[$p$-adic $QR$-factorization] \label{cor: padic QR}
		Let $A \in \mbz_p^{n \times m}$ be a matrix. Then there exists a $Q \in \GL_n(\mbz_p)$ and an upper triangular matrix $R \in \mbz_p^{n \times m}$ such that $A = QR$. 
	\end{corollary}

	We decided to refer to the factorization above as a $QR$-factorization since many applications are similar to the real case. However, one important difference is that $Q$ is generally not an element of $\operatorname{O}_n(\mbq_p)$. As in the case over $\mbr$, an elementary proof of Corollary~\ref{cor: padic QR} can be given by exhibiting an algorithm to compute a $QR$-decomposition. The algorithm is well-known, and in fact it is just the algorithm to compute a $PLU$-decomposition, though with each row-pivot chosen by taking the vector with the largest $p$-adic norm \cite[Chapter 4]{Kedlaya2010differential}. We also note since $Q \in \GL_n(\mbz_p)$, that $R$ can be chosen to be the Hermite normal form of $A$. We summarize the discussion as:
	
	\begin{remark}
		The $p$-adic $QR$-decomposition of $A$ is computed using a $PLU$-decomp\-osition, with pivots chosen by $p$-adic norm.
	\end{remark}
	
	An essential property of the $QR$-factorization of a real matrix $A$ is that $Q$ has a well-conditioned inverse. The same is true $p$-adically. The columns of $Q$ in the $p$-adic case are also ``orthogonal'' in the sense of \cite[Section 50]{Schikhof1984ultrametric}. Unfortunately, this notion of orthogonality does not help to compute the inverse of $Q$ like in the real case. As we might expect, $QR$-factorizations give us a way to compute the singular value decomposition of a matrix in the $p$-adic setting.
	
	\begin{proposition}
		If $A \in M_n(\mbq_p)$, then there exist $U,V \in \GL_n(\mbz_p)$ and a diagonal matrix $\Sigma \in M_n(\mbq_p)$ such that $A = U\Sigma V^T$.
	\end{proposition}
	
	\begin{proof}
		Let $A = URP'$ be a $QR$-decomposition with column pivoting. Now let $R^T = V\Sigma$ be a $QR$-decomposition. The result follows.
	\end{proof}
	
	The singular value decomposition is the Smith normal form of the $\mbz_p$-module generated by columns of $A$ \cite[Chapter 4]{Kedlaya2010differential}. Since the matrices $U,V$ are well-conditioned, we can use a singular value decomposition to stabilize pnumerical computations (as in the real setting).

	\subsection{The eigenvector problem: semantics}
		
		Let $A \in M_n(\mbz_p)$ be a matrix. In the setting of a finite $p$-adic precision computation, we may formulate two versions of the eigenvector problem:
		
		\begin{problem}[Eigenvector problem, Version~I]
			 Let $(\lambda, v)$ be an exact eigenpair for $A$ over $\mbz_p$. Given an approximation $\wtilde A$ to $A$ such that $A = \wtilde A + O(p^N)$, compute a pair $(\wtilde \lambda, \wtilde v)$ such that $\norm{\wtilde v} = 1$ and $\wtilde A \wtilde v = \wtilde \lambda \wtilde v + O(p^N)$. 
		\end{problem}
		
		\begin{problem}[Eigenvector problem, Version~II]
			Let $(\lambda, v)$ be an exact eigenpair for $A$ over $\mbz_p$. Given an approximation $\wtilde A$ to $A$ such that $A = \wtilde A + O(p^N)$, compute a triple $( \wtilde \lambda, \wtilde v, M) \in \mbz_p \times \mbz_p^n \times \mbz$ such that:
				\begin{enumerate}[(i)]
					\item
					$\norm{\wtilde v}=1$, $\lambda = \wtilde \lambda + O(p^M)$ and $v = \wtilde v + O(p^M)$,
					\item
					for any $B = 0 + O(p^N)$, we have that $(\tilde A + B)\wtilde v = \wtilde \lambda \wtilde v + O(p^M)$,
					\item
					$M$ is maximal among such triples.
				\end{enumerate}
		\end{problem}
	
		In Version~II, the second criteria ensures that the problem is well-posed.
		
		\begin{lemma}
			Let $A \in M_n(\mbz_p)$, let $\wtilde A$ be an approximation such that $A = \wtilde A + O(p^N)$, and let $(\lambda,v,M)$ satisfy the first two criteria for a solution to Version~II. Then $M \leq N$. 
		\end{lemma}
	
		\begin{proof}
			If $\norm{(A-\wtilde A)\wtilde v}_p \geq p^{-N}$ then clearly $M \leq N$. Otherwise, choose any $B \in M_n(\mbz_p)$ such that $\norm{B\wtilde v}_p = 1$. Then $\norm{ (A - \wtilde A + p^N B)\wtilde v }_p = p^{-N}\norm{B\wtilde v}_p = p^{-N}$.
		\end{proof}
	
		The following example demonstrates the difference between Version~I and Version~II of the problem.
		
		\begin{example} \label{example: precision is important}
			Let $p=7$, and consider the matrices
				\[
					A := 
					\begin{bmatrix}
						p^3 & 1 \\ 0 & -p^3
					\end{bmatrix},
					\qquad
					B := 
					\begin{bmatrix}
					p^3 & 1 \\ p^6 & -p^3
					\end{bmatrix},
				\]
			Both matrices are rounded to the same result with precision $N := 6$. However, the respective eigenvectors are
			\[
				v_A := \begin{bmatrix} 1 \\ 0 \end{bmatrix}, \qquad v_B := \begin{bmatrix} 1 \\ p^3 \sqrt{2} \end{bmatrix}.
			\]
			Note in $\mbz_7$ that $\sqrt{2} = \pm 4 + O(p)$. We see that the equation $Av_A + O(p^6) = p^3v_A + O(p^6) = Bv_A + O(p^6)$. However, $v_A \neq v_B + O(p^6)$. 
		\end{example}
		
		
		Version~II exhibits pathological behavior whenever the characteristic polynomial of $A$ is not square-free modulo $p^M$. For instance, as $X$ ranges over all $\mbz_p^n$ matrices, we see that the eigenvectors of $I + p^MX$ could be anything. In other words, any solution to Version~II of the problem for $I + p^MX$ has $M \leq 0$. A similar difficulty arises in the eigenvector problem for real matrices. Because of this pathology, we restrict our attention to solving Version~I. 
		
		Of course, we can give analogous formulations for computing a Jordan form of $A$ or for computing a block Schur form for $A$. Generally, we will work with Version~I of these problems as well.

	\subsection{Commentary on the classical algorithm}
		
		The naive algorithm to compute the eigenvalues of a matrix relies on the calculation and factorization of the characteristic polynomial. In the numerical setting over $\mbr$ it is desirable to avoid this step due to the instability of solving for the roots of a polynomial. Also $p$-adically, there is also a loss of precision in solving for the roots of a polynomial. Thus, in order to accurately compute the nullspace of $A - \lambda I$ the roots of the characteristic polynomial, and thus the characteristic polynomial itself, must be known to as much precision as possible. 
		Presently, there are two algorithms to compute the characteristic polynomial at the maximum possible precision, both of which use more than $O(n^3)$ arithmetic operations (asymptotically):
		\begin{enumerate}
			\item 
			Use a division-free calculation. The asymptotic run-time is at least $O(n^4)$.
			\item
			Use the algorithm of \cite{CRV2017characteristic}. Compute the characteristic polynomial of a matrix with inflated precision using division, then truncate to the optimal precision computed from the precision lattice analysis. The runtime is $O(\rho n^3)$, where $\rho$ represents the factor of performing arithmetic at the higher precision. Practically, the computation of the comatrix and the precision increase required for the precision analysis appears to impose a large cost \cite[Remark 5.3]{CRV2017characteristic}.
		\end{enumerate} 
		
		In the next subsection, we discuss how to avoid an expensive computation of the characteristic polynomial for most matrices in $M_n(\mbz_p)$. 
		
		\begin{remark} \label{rem: most meaning}
			To clarify the term ``most'', if $A + O(p^N)$ is an $n \times n$-matrix whose entries are chosen with the uniform probability distribution on $[0, \ldots, p^N-1]$, then the limit as $n \ra \infty$ of the probability that $\chi_{A}$ is square-free is at least $\frac{1-p^{-5}}{1+p^{-3}}$ \cite{Fulman2002random}.
		\end{remark}
	
	\vfill
	\pagebreak
	\subsection{The power iteration algorithm}
		\
		
		\begin{minipage}{\textwidth}
			\bigskip
		\begin{algorithm} \label{algo: eigenvector-main} 
			Eigvecs($A$)
			\begin{algorithmic}[1]
				\REQUIRE $A + O(p^N)$, an $n \times n$-matrix over $\mbz_p$. 
				\ENSURE The eigenpairs $(\lambda_i, v_i)$ of the matrix $A$.
				
				\bigskip
				\IF {$A$ is diagonal }
				\RETURN $\{ (a_{11}, e_1), \ldots , (a_{nn}, e_n) \}$.
				\ENDIF
				
				\IF {$A \mod p \equiv 0$} 
					\STATE $\nu = \min_{i,j} \{ \ord x_{ij} \}$
					\STATE \label{step: recursive call for pA} $V$ = Eigvecs($p^{-v} X$)
					\STATE \label{step: reset prec} Reset precision of $V$
					\RETURN $V$
				\ENDIF

				\bigskip
				\STATE 
				Compute $\chi_{A,p}$
				\IF {$\chi_{A,p}$ has no linear factors}
					\RETURN Nothing.
				\ENDIF		
				
				\IF { $\chi_{A,p} \not \equiv (x-\lambda)^n \pmod p$}
					\STATE $\{ (X_1,V_1), \ldots, (X_r,V_r) \}$ := PowerIterationDecomposition($A$, $\chi_{A,p}$)
					\RETURN Eigvecs($X_i$) for $i=1, \ldots, r$\\
				\ELSE
				\RETURN \label{step: nilpotent block} ClassicalAlgorithm($A$)
				\ENDIF
			\end{algorithmic}
		\end{algorithm}
		\bigskip
		\end{minipage}
		
		Note in step~\ref{step: reset prec}, we may reset the precision since the input matrix on the previous step $X$ was a multiple of $p^\nu$. If $p$ is moderately large, then for most matrices the size of the Jordan blocks modulo $p$ will be small (see Remark~\ref{rem: most meaning}). The square matrix in step~$13$ is expected to be small, so we apply the costly method to compute the characteristic polynomial of this block.

		\begin{minipage}{\textwidth}
		\bigskip
		\begin{algorithm} PowerIterationDecomposition($A$, $\chi_{A,p}$)
			\begin{algorithmic}[1]
				\REQUIRE
				\ \\
				$A + O(p^N)$, an $n \times n$-matrix over $\mbz_p$. \\
				$\chi_{A,p}$, the characteristic polynomial of $A \pmod p$.
				\bigskip
				\ENSURE
					A list of pairs of matrices $\{ (X_i, V_i) \}_i$ such that $A V_i = V_iX_i + O(p^N)$.
				\bigskip
				\STATE
				Compute eigenpairs $(\lambda_i, V_i)$ mod $p$
				\STATE
				Set $m$ to be the largest multiplicity of a linear factor of $\chi_{A,p}$.
				\FOR {$V = V_1, \ldots, V_r$ }
				\STATE 
				\label{step: lift step}
				Lift the eigenpair to an approximate eigenpair $(\lambda_i, V_i)$ in $\mbz_p$
				\STATE
				Set $B = (A - \lambda_i I)$.  
				\STATE
				\textbf{Iterate} for $\log_2 (mN)$ times, $B = B^2$
				\STATE
				Set $V$ := nullspace($B + O(p^N)$).
				\STATE \label{step: solve deflation}
				Solve $VX_i = AV$. 
				\ENDFOR
				\RETURN $\{ (X_1,V_1), \ldots, (X_r,V_r) \}$.
			\end{algorithmic}
		\end{algorithm}
		\bigskip
		\end{minipage}
		
		To clarify, the lifting in step~\ref{step: lift step} is just the trivial operation of interpreting the elements $\{ 0, \ldots, p-1 \}$ as elements of $\mbz_p$. 
		
		\subsection{Proof of correctness}
		
		We give a quick proof of correctness of the power iteration algorithm.
		
		\begin{lemma}
			Let $\lambda$ be an approximate eigenvalue of $A \pmod p$, let $m$ be the multiplicity of $\lambda$ as a root of $\chi_{A,p}$, and let $M$ be the nullspace of $(A-\lambda I )^{Nm}$. Then $M$ is a free $\mbz/p^N\mbz$-module with trivial annihilator. Letting $V$ be a matrix whose columns are the generators of $M$, there exists an $X \in M_N(\mbz_p)$ with such that $AV = VX + O(p^N)$.
		\end{lemma}
		
		\begin{proof}
			
			First, note that $(A-\lambda I) \pmod p$ has $0$ as an eigenvalue with multiplicity $m$. Ordering the singular values and eigenvalues by size, we have that $(A-\lambda I)$ satisfies
				\[
					\abs{\lambda_1 \ldots \lambda_{n-m}} = \abs{\sigma_1\ldots \sigma_{n-m}}_p = 1 \quad \text{ and } \quad \abs{\lambda_{n-m}}_p > \abs{\lambda_{n-m+1}}_p.
				\]
			Thus, by the Hodge-Newton decomposition \cite[Theorem 4.3.11]{Kedlaya2010differential},  there is a matrix $U \in \GL_n(\mbz_p)$ such that
				$B := U^{-1}(A-\lambda I)U$
			is block upper triangular, with the lower right block $B_{22}$ accounting for the $m$ small eigenvalues and the upper left block $B_{11}$ accounting for the unit eigenvalues. Now, we see that $B_{22}^m = 0 \pmod p$, so
				\[
					(B^m)^N =
					\begin{bmatrix}
						B_{11}^{Nm} & * \\
						0 & 0
					\end{bmatrix}
					+ O(p^N).
				\]
			But as $U \in \GL_n(\mbz_p)$, we see that the Smith normal form of $(A-\lambda I)^{Nm}$ is 
					$\begin{bmatrix}
						I & 0 \\
						0 & 0
					\end{bmatrix}.$			
			Thus, the kernel of $(A-\lambda I)^{Nm}$ is free as a $\mbz/p^N\mbz$-module with trivial annihilator in $\mbz/p^N\mbz$. Next, note that we have
				\[
					(A-\lambda I)^{Nm}(AM) = A (A-\lambda I)^{Nm} V = 0
				\]
			so $AM \ssq M$.  Thus $A$ restricted to $M$ defines an endomorphism of $M$, represented by a square matrix.
		\end{proof}
		
		As an immediate corollary, we have:
		
		\begin{corollary} \label{cor: power iteration output}
			Let $(V_1, X_1), \ldots, (V_r,X_r)$ be the blocks returned from power iteration. With $\mathbf{V} := [V_1, \ldots, V_r]$ the horizontal join, we have
				\[
					A \mathbf{V} = \mathbf{V}
					\begin{bmatrix}
						X_1 \\& X_2 \\
						& & \ddots \\
						& & & X_r
					\end{bmatrix}
					+ O(p^N).
				\]
		\end{corollary}
	
	
		We hope that step~\ref{step: nilpotent block} can be improved, as we believe it is helpful in solving polynomial systems whose solutions have large valuations. Thus, we raise the question:
		
		\begin{question}
			Given a topologically nilpotent matrix $B \in M_n(\mbz_p)$, is there an efficient iterative strategy to determine the eigenvectors of $B + O(p^N)$?
		\end{question}

		\subsection{The Schur form algorithm}
				
		Solving for the nullspace at the end of the iteration is the dominant cost in the power iteration algorithm. In the classical numerical setting, we can avoid this drawback by using the $QR$-iteration algorithm. We also have a $p$-adic version of the $QR$-algorithm. With the $p$-adic $QR$ step replacing the usual $QR$ factorization, our algorithm is simply the original $LR$-algorithm \cite{rutishauser1958eigenvalue} that inspired the $QR$-algorithm for real matrices. The $LR$-algorithm fell out of favor some time in the 1960's since it is numerically less stable than the $QR$-algorithm, but $p$-adically this situation is reversed!
		
		
		We now give the algorithm.
		
		\begin{minipage}{\textwidth}
			\bigskip
			
			\begin{algorithm} QR-Iteration($A$, $\chi_{A,p}$)
				\begin{algorithmic}[1]
					\REQUIRE
					\ \\
					$A + O(p^N)$, an $n \times n$-matrix over $\mbz_p$. \\
					$\chi_{A,p}$, the characteristic polynomial of $A \pmod p$.
					\bigskip
					
					\ENSURE
					A (block) triangular form $T$ for $A$, and matrix $V$ such that $AV = VT + O(p^N)$. 
					
					\bigskip
					\STATE
					Set $\lambda_1, \ldots, \lambda_\ell$ to be the roots of $\chi_{A,p}$ in $\mbf_p$, lifted to $\mbz_p$.
					\STATE
					Set $m_1, \ldots, m_\ell$ to be the multiplicities of the roots of $\chi_{A,p}$.
					
					\STATE
					Compute $B,V$ such that $AV = VB$ and $B$ is in Hessenberg form.
					
					\FOR {$i= 1, \ldots, \ell$}
					\FOR {$j= 1, \ldots , m_iN$}
					\STATE
					Factor $(B-\lambda_i I) = QR$
					\STATE
					Set $B := RQ + \lambda_i I$
					\STATE
					Set $V := Q^{-1}V$						
					\ENDFOR
					\ENDFOR
					\RETURN $B, V$.
				\end{algorithmic}
			\end{algorithm}
			\bigskip
		\end{minipage}

		\pagebreak
	
		\begin{minipage}{\textwidth}
			\bigskip
			\begin{algorithm} \label{algo: schur-main} 
				BlockSchurForm($A$)
				\begin{algorithmic}[1]
					\REQUIRE
						$A + O(p^N)$, an $n \times n$-matrix over $\mbz_p$.
					\ENSURE 
						A block upper triangular matrix $T$ with $\ell+1$ distinct blocks and matrix $V$ such that $AV = VT + O(p^N)$. Here, $\ell$ is the number of linear factors of $\chi_{A,p} = \chi_A \pmod p$.
						
					\bigskip
					\IF {$A$ is diagonal }
					\RETURN $A$, $\{ e_1, \ldots , e_n \}$.
					\ENDIF
					
					\IF {$A \mod p \equiv 0$} 
					\STATE $\nu = \min_{i,j} \{ \ord_p x_{ij} \}$
					\STATE \label{step: schur recursive call for pA} $V$ = BlockSchurForm($p^{-v} X$)
					\STATE \label{step: schur reset prec} Reset precision of $V$
					\RETURN $V$
					\ENDIF
					
					\bigskip			
					\STATE 
					Compute $\chi_{A,p} := \chi_A \pmod p$
					\IF {$\chi_{A,p}$ has no linear factors}
						\RETURN $A$, $\{ e_1, \ldots , e_n \}$
					\ENDIF
					
					\bigskip					
					\IF { $\chi_{A,p} \neq (x-\lambda)^n$}						
						\RETURN $B,V$ := QR-Iteration($A$, $\chi_{A,p}$) 
					\ELSE {}
						\STATE
						$B,V$ := $A$, $\{ e_1, \ldots , e_n \}$
					\ENDIF
					
					\bigskip
					\FOR {$j := 1 \ldots \ell$}
						\IF { $\chi_{B_{ii},p} \equiv (x-\lambda)^m$ for $m > 1$ }
							\STATE \label{step: schur nilpotent block}
							$C_j, X_j$ :=  ClassicalAlgorithm($B_{j}$)
							\STATE
							Update $B$ and $V$
						\ENDIF
					\ENDFOR	
					\RETURN $A,V$				
				\end{algorithmic}
			\end{algorithm}
		\end{minipage}

		We briefly comment on some aspects of this algorithm analogous to the archimedean setting (the proofs for the various statements are analogous as well). Note that each $p$-adic $QR$-iteration preserves the Hessenberg form, so the cost of applying the $QR$-step is $O(n^2)$. From the discussion of the $LR$-algorithm in \cite[Sections 5-9]{wilkinson1965convergence}, we see that the subdiagonal entry $(i,i+1)$ converges to $0$ at a rate of $O\left(\frac{\abs{\lambda_{i+1}}_p^s}{\abs{\lambda_i}_p^s} \right)$, provided that $\abs{\lambda_{i+1}}_p < \abs{\lambda_i}_p$. For $B \in M_n(\mbz_p)$, if $\abs{\lambda_{i+1}}_p < \abs{\lambda_i}_p = 1$, then $\frac{\abs{\lambda_{i+1}}_p}{\abs{\lambda_i}_p} < p^{-\frac{1}{m}}$, where $m$ is the largest multiplicity of an eigenvalue of $B \pmod p$. Additionally, if some subdiagonal entries are identified to be zero during the iteration, the standard deflation strategies can be used to accelerate the algorithm.
				

		\begin{remark}
			In our version of the iteration we use a constant shift factor, and so the iteration converges in at most $m N$ steps \cite{wilkinson1965convergence}. We believe that a $p$-adic analogue of the Rayleigh quotient strategy to refine the eigenvalue approximation will lead to a drastic speed-up in convergence.
		\end{remark}
		
		\begin{remark}
			It is possible to use the $QR$-iteration to compute the valuations of the eigenvalues iteratively. After applying a single round of the $QR$-iteration to a Hessenberg matrix $A$ with no shift (i.e, $\lambda_i=0$), the valuations of the eigenvalues can be read from the diagonal blocks in the resulting matrix. If $B_{ii}$ is one of the diagonal blocks with no subdiagonal entry equal to $0$, then all of the eigenvalues of $B_{ii}$ must have the same valuation, as otherwise some of the subdiagonal entries of $B_{ii}$ would have converged to $0$. In fact, the valuations of the eigenvalues of $A$ can be computed this way even if some of the eigenvalues are not defined over $\mbq_p$.
		\end{remark}

		\subsection{Complexity analysis}
		
		We denote by $M(n)$ the cost of matrix multiplication.
		
		\begin{proposition}
			Let $N$ be the precision and let $m$ be the largest size of a Jordan block of $A \pmod p$. Then the number of $\mbq_p$-arithmetic operations performed by Algorithm~\ref{algo: eigenvector-main} is as most
			\[
				O\left( \rho m^3 + \ell (n^3 + M(n) \log_2(mN)) \right).
			\]
		\end{proposition}
	
		\begin{proof}
		
		We tabulate the costs in a table.
		\begin{center}
		\begin{tabular}{l|l|l}
			Line(s) & Cost per line \\ \hline \hline			
			Eigvecs & \\ \hline 
			3 to 7 & Negligible \\
			8 & $O(M(n) (\log n)(\log \log n) )$ & Finite field operations. \\ \hline 
			12 & $O( \ell n^3 + M(n) \log_2 mN)$ & Power iteration \\
			15 & $O(\rho m^3)$ & Classical algorithm \\
			Total: & $O( \ell n^3 + \rho m^3 + M(n) \log_2 mN)$ \\ \\ \hline \hline
			
			Power iteration decomposition & \\  \hline
			1 & $O(n^3)$ & Finite field operations \\
			3 & $\ell$ iterations \\
			-- 5 & $O( M(n) \log_2 mN)$ \\
			-- 6 & $O(n^3)$ \\
			7 & $O(m^3)$ \\ \hline
			Total: & $O( \ell n^3 + M(n) \log_2 mN)$. 		
			
		\end{tabular}
		\end{center}
	
		\end{proof}
	
		The $\rho m^3$ term corresponds to blocks which cannot be decomposed using power iteration, i.e, where the classical algorithm is invoked. Note that the cost of computing $\chi_A \pmod p$ is negligible, since over a finite field the characteristic polynomial can be computed in $O(M(n)(\log n)(\log \log n))$ finite field operations \cite{Storjohann2001frobenius}.
	
		If the linear factors occur with small multiplicities mod $p$, our algorithm significantly outperforms the naive strategy. If all linear factors mod $p$ occur with multiplicity $1$ the classical algorithm is never called, so the asymptotic run-time simplifies to ${O(n^3 + \ell M(n) \log_2(N))}$. The worst-case scenario of the algorithm is if $\chi_A \pmod p \equiv (x - \alpha)^n$, in which case the default algorithm is called immediately; the only wasted time in this case is the inexpensive computation of $\chi_A \pmod p$.
		
		We now study the complexity of the block Schur form algorithm.
		
		\begin{proposition}
			Let $N$ be the precision, let $m$ be the largest size of a Jordan block of $A \pmod p$, and let $\ell$ be the number of linear factors of $\chi_A \pmod p$. Then the block Schur form can be computed in $O(n^3 + \ell  mN n^2 + \rho m^3)$ $\mathbb{Q}_p$-arithmetic operations.
		\end{proposition} 
	
		\begin{proof}
				We tabulate the costs in a table.
			
			\begin{center}
			\begin{tabular}{l|l|l}
				Line(s) & Cost per line  \\ \hline \hline			
				BlockSchurForm & \\ \hline
				8 to 11 & Negligible \\
				12 & $O(n^3 + \ell mN n^2)$ & $QR$-iteration \\
				15 & $\ell$-iterations \\
				-- 17 & $O(\rho m^3)$ & Classical algorithm \\
				-- 18 & $O(n m^2)$ \\ 
				Total: & $O(\rho m^3 + n^3 + \ell mN n^2)$ \\ \\ \hline \hline
				
				$QR$-Iteration. & \\ \hline
				1 & $O(M(n)(\log n)(\log \log n))$ & Finite field operations \\
				2 & $O(1)$ \\
				3 & $O(n^3)$ & Hessenberg form\\
				4 & $\ell$-iterations \\
				-- 5 & $N$-iterations \\
				-- -- 6 & $O(n^2)$ \\
				-- -- 7 & $O(n^2)$ & (parallel with step $6$ row operations) \\
				-- -- 8 & $O(n^2)$ & (parallel with step $6$ row operations) \\
				Total: & $O(n^3 + \ell m Nn^2)$

			\end{tabular}
		\end{center}
		\end{proof}

	\subsection{$p$-adic Householder reflections} \label{subsec: householder}
	
	We close the section with a quick observation about Householder reflections.

	\begin{lemma} \label{lem: p-adic householder}
		Let $p$ be an odd prime. Let $x \in \mbz_p^n$ be a vector with exactly one coordinate with minimal valuation $r$. We further assume this coordinate is not $x_1$. Let $e_1$ be the first standard basis vector. Then:
		\begin{enumerate}[(a)]
			\item
			$x^Tx$ is a non-zero square in $\mbz_p$.
			\item
			Let $\alpha := \sqrt{x^Tx} \in \mbz_p$ be one of the square roots and let $v' := x - \alpha e_1$. Choose (the unique) $v \in \mbz_p^n$ such that $v' = p^rv$. Then $\abs{v^Tv}_p = 1$, and the Householder transformation
			\[
			H := I - \frac{2}{v^Tv} vv^T
			\]
			is a well-defined reflection. In particular $H \in \Ogroup_n(\mbz_p)$.
			\item
			$Hx = \alpha e_1$.
		\end{enumerate}
	\end{lemma}
	
	\begin{proof}
		\begin{enumerate}[(a)]
			\item
			Assume $x_i$ is the coordinate with minimal valuation. Write $x = p^ry$. We have that $x^Tx = p^{2r}(y^Ty) = p^{2r}(y_i^2 + O(p))$. As $y_i$ is a unit, we see by Hensel's lemma that this is a square.
			
			\item
			Assume again $x_i$ is the coordinate with minimal valuation. Note that
				$\alpha^2 = \sum_{j=1}^n x_j^2$,
			so as $x_i$ is the unique coordinate with minimal valuation, we have $\alpha = x_i + O(p^{r+1})$. Now
			\begin{align*}
				p^{2r}v^Tv &= (x^Tx - 2\alpha x^Te_1 + \alpha^2) \\
				&= 2( x^Tx - (x_i + O(p^{r+1}))x_1) \\
				&= 2( x_i^2 - x_i x_1 ) + O(p^{2r+1}) \\
				&= 2x_i( x_i - x_1 ) + O(p^{2r+1})
			\end{align*}
			By our assumption on $x$, we have $x_i-x_1 = x_i + O(p^{r+1})$. Thus, $v^Tv$ is a unit and $H$ is well-defined. That it is a reflection follows from the usual calculation, which is
			\begin{align*}
				HH &= \left(I - \frac{2}{v^Tv} vv^T \right)\left(I - \frac{2}{v^Tv} vv^T \right) \\
				&= I - \frac{4}{v^Tv} vv^T + \frac{4}{ (v^Tv)^2 } (vv^T)(vv^T) \\
				&=  I - \frac{4}{v^Tv} vv^T + \frac{4}{ (v^Tv)} (vv^T) \\
				&= I.
			\end{align*}
			Finally, it is clear that all of the entries of $H$ are integral.
			\item
			The result is clear from direct calculation.
		\end{enumerate}
	\end{proof}

	We now offer a Householder version of the Hessenberg algorithm. Our version of the Hessenberg algorithm is more a theoretical curiosity; for our practical implementations, we use  \cite[Algorithm 1]{CRV2017characteristic} as it appears to be more efficient.
	
	\begin{minipage}{\textwidth}
		\bigskip
		\begin{algorithm} \label{algo: hessenburg} 
			HessenbergForm($A$)
			\begin{algorithmic}[1]
				\FOR { $j=1,\ldots, (n-1)$}
					\IF{ The $j$-th subdiagonal is 0}
						\STATE
							\textbf{continue}
					\ENDIF
					\STATE
						Swap the last (i.e, $n$-th) row with row $i$ to ensure $a_{nj}$ has the minimal valuation of the sub-diagonal elements
					\STATE
						Add multiples of the last ($n$-th row) to the $i$-th row, for $j < i < n$, to ensure $a_{nj}$ is the unique sub-diagonal element with minimal valuation
					\STATE
						Apply the corresponding column operations to preserve similarity. None of these operations affect columns $1, \ldots, j$				
					\STATE
						Apply the Householder reflection to set the $j$-th subcolumn to a multiple of $e_1$.
				\ENDFOR
			\end{algorithmic}
		\end{algorithm}
		\bigskip
	\end{minipage}

\section{Polynomial system solving} \label{sec: solve equations}

	\newcommand{\N}{\mathcal{N}}
	\newcommand{\Res}{\mathrm{Res}}
	\newcommand{\Span}{\mathrm{Span}}
	\newcommand{\id}{\mathrm{id}}

	In this section, we solve 0-dimensional systems of polynomial equations over $\mbq_p$. The general method is based on the truncated normal form solver of \cite{TMV2018truncated}. We begin by reviewing some of their terminology. Throughout, we let $R := \mbq_p[x_1, \ldots, x_n]$ and let $I$ be an ideal such that $d = \dim_{\mbq_p} R/I < \infty$.

	\begin{definition}
		A \emph{normal form} on $R$ with respect to $I$ is a linear map $\N : R \ra B$ where $B \ssq R$ is a $\mbq_p$-vector subspace of dimension $d$ such that the sequence
		\[
			\xym{ 0 \ar[r] & I \ar[r] & R \ar[r]^\N & B \ar[r] & 0 }
		\]
		is exact, and $\N |_B = \id_B$.
	\end{definition}

	Consider the canonical exact sequence
	\[
		\xym{ 0 \ar[r] & I \ar[r] & R \ar[r]^\pi & R/I \ar[r] & 0 }.
	\]
	Let $s\: R/I \ra R$ be a section of $\mbq_p$-vector spaces and let $B$ be the image of $R/I$. We can construct a normal form by setting $\N := \pi \circ s$, giving the commutative diagram
	\[
	\xym{
		& & B \ar@{^(->}[d] & \\
		0 \ar[r] & I \ar[r]^-\iota & R  \ar[r]^-\pi &  R/I \ar[r] \ar[ul]_-s & 0.
	}
	\]
	Vice-versa, any normal form $\N$ defines a section of $\pi$. 
	
	\bigskip
	Following \cite{TMV2018truncated}, there is no need to compute an entire normal form for our purposes, and we can often restrict to a finite dimensional subspace.
	
	\begin{definition}
		Let $B \ssq V \ssq R$ with $B, V$ finite dimensional $\mbq_p$-vector subspaces such that $x_i B \ssq V$ for all $i = 1, \ldots, n$, and $\dim_{\mbq_p} B = \dim_{\mbq_p} (R/I)$. A \emph{Truncated Normal Form (TNF)} on $V$  with respect to $I$ is a linear map $\N : V \ra B$ such that $N$ is the restriction to $V$ of a normal
		form with respect to $I$. That is, the sequence
		\[
			\xym{0 \ar[r] & I \cap V \ar[r]^-\iota & V  \ar[r]^-\N &  B \ar[r]  & 0}
		\]
		is exact, and $\N |_B = \id_B$.
	\end{definition}

	Similar to normal forms, from a commutative diagram of $\mbq_p$-vector spaces
	\[
	\xym{
		& & B \ar@{^(->}[d] & \\
		0 \ar[r] & I \cap V \ar[r]^-\iota & V  \ar[r]^-\pi &  V/I \cap V \ar[r] \ar[ul]_-s & 0
	}
	\]
	we obtain a truncated normal form by $\N = s \circ \pi$, where $s$ is a section of $\pi$ and $x_iB \ssq V$ for all $i$.
	
	One way to construct a truncated normal form is via resultant matrices.	
	\begin{definition}
		Let $k$ be a field and let $f_1, \ldots, f_r \in k[x_1, \ldots, x_n]$ be polynomials. Let $D$ be a positive integer and let $d_i := \deg f_i$. A \emph{resultant matrix} is a $k$-linear map of the form
			\[
			\begin{tabu}{rccc}
				\Res\:& V_{D-d_1} \times \ldots \times V_{D-d_r} &\ra& V_D \\
					& (q_1, \ldots, q_r) &\mapsto& f_1q_1 + \ldots + f_rq_r
			\end{tabu}
			\]
		where $V_d$ is the $k$-vector space of polynomials of degree at most $d$. (In particular, $V_{-n} = \{0\}$ for $n$ a positive integer.) We denote the corresponding matrix as $[\Res_{ij}]$.
	\end{definition}
	
	Note that the image of the resultant map is contained in $V_D \cap I$. If $D$ is sufficiently large we obtain a truncated normal form. (The result below is quite old, and originally due to Macaulay.)
	
	\begin{proposition}
		If $D = \sum_{i=1}^{r} (d_i-1) + 1$, then $\Im \Res = I \cap V_D$. In particular, the canonical quotient $\pi\: V_D \ra V_D/(V_D \cap I)$ is represented by the matrix  $(\ker [\Res_{ij}]^T)^T$.
	\end{proposition}

	\begin{proof}
		See \cite[Section 4]{TMV2018truncated}.
	\end{proof}
	
	\subsection{The algorithm}
	
	We now summarize our modifications of \cite[Algorithm 1]{TMV2018truncated}.  Our implementation is available at:
	\begin{center}
		\url{https://github.com/a-kulkarn/pAdicSolver}
	\end{center}
	
	The critical theoretical underpinning of \cite[Algorithm 1]{TMV2018truncated} is Stickelberger's Theorem \cite[page 621]{coxlittleoshea2015ideals}.
	
	\begin{theorem}
		Let $k$ be a field, let $R := k[x_1, \ldots, x_n]$, and let $I$ be an ideal such that $\dim_k R/I$ is finite. Let $p_1, \ldots, p_\ell$ be the geometric points of the affine scheme $\operatorname{Spec}(R/I)$ with multiplicities $m_1, \ldots, m_\ell$ respectively. Let $[f]\: R/I \ra R/I$ be the endomorphism of $R/I$ induced by multiplication by $f \in R$. Then the eigenvalues of $[f]$ and their multiplicities are
		$(f(p_1), m_1), \ldots, (f(p_\ell), m_\ell)$. The $i$-th invariant subspace of $[f]$ is $\{ g \in R/I : g(p_j) = 0 \text{ for } j \neq i \}$. 
	\end{theorem}

	\begin{proof}
		Let $k^\textrm{al}$ be an algebraic closure for $k$. We may assume that $R/I \otimes_k k^\textrm{al}$ is local by the Chinese Remainder Theorem. Note $R/I \otimes_k k^\textrm{al}$ is Artinian, so for $\mathfrak{m}$ the maximal ideal we have $\mathfrak{m}(\mathfrak{m}^r) = \mathfrak{m}^r$ for some $r > 0$. By Nakayama's lemma we have $\mathfrak{m}^r = 0$. Thus, for any $f \in \mathfrak{m}$, $[f]$ is nilpotent and so has $0$ as an eigenvalue with multiplicity $\dim_{k^\textrm{al}} (R/I \otimes_k k^\mathrm{al})$. Since $R/I \otimes_k k^\textrm{al} = k^\textrm{al} \oplus \mathfrak{m}$ as a $k^\textrm{al}$-vector space, we now need only check the result for constant functions. For $f = 1$, the result is clear.
	\end{proof}
	
	Note that if $k$ is not algebraically closed, the eigenvalues and invariant subspaces may only be defined over a finite extension of $k$; these correspond to points whose coordinates lie in a finite extension of $k$. Critically, the eigenvalues of $[x_i]$ are the $i$-th coordinates of the geometric points. We now give the $p$-adic solver algorithm:
	
	\begin{minipage}{\textwidth}
		\bigskip
		\begin{algorithm} SolveQpSystem($f_1, \ldots, f_m$)
		\begin{algorithmic}[1]
		\REQUIRE
			Polynomials $f_1, \ldots, f_m \in \mbq_p[x_1, \ldots, x_n]$ defining a $0$-dimensional polynomial system.
		\ENSURE
			A set of approximate solutions to the distinct $\mbq_p$-solutions.
			
		\bigskip
		\STATE
			Construct the resultant map $[\Res_{ij}]$ with $D$ sufficiently large
			
		\STATE \label{algo: polysolve: kernel step}
			Compute $[\pi_{ij}] = (\ker [\Res_{ij}]^T)^T$
		\STATE
			Extract the submatrix $M$ of $[\pi_{ij}]$ whose columns correspond to monomials $m$ such that $\deg{m} < D$.
			
		\STATE \label{algo: polysolve: iwasawa step}
			Compute a square subblock of $M$ that is pnumerically well conditioned, using the $p$-adic $QR$-factorization. Let $\mathbf{b}$ be the basis corresponding to this sub-block.
			
		\STATE
			Construct the multiplication by $x_i$ matrices $[x_1]_\mathbf{b}, \ldots, [x_n]_\mathbf{b}$ from the columns of $[\pi_{ij}]_\mathbf{b}$.
			
		\STATE \label{eigenalgorithm changed}
			Compute the eigenvectors of a random linear combination via iteration.
	
		\RETURN 
			The eigenvalues for $[x_1]_\mathbf{b}, \ldots, [x_n]_\mathbf{b}$ indexed by the invariant subspaces.
			
		\end{algorithmic}
		\end{algorithm}
		\bigskip
	\end{minipage}

	By $[x_i]_{\mathbf{b}}$, we mean the matrix corresponding to the multiplication-by-$x_i$ operator in $\mathbf{b}$-coordinates. The choice of $B := \Span(\mathbf{b})$ in step~\ref{algo: polysolve: iwasawa step} defines a subspace of $V_D$ of the same dimension as $R/I$, hence a section $s\: R/I \ra V_D$ and thus a truncated normal form. Since $[x_i] B \ssq V_D$, we can easily read off the action of the operator $[x_i]\: B \ra B$ from the columns of $[\pi_{ij}]$. For explicit details see \cite[Theorem 4 and following text]{TMV2018truncated}.

	The differences between the $\mbq_p$ and $\mbr$ versions of the algorithm are precisely the adaptations in the linear algebra, i.e, steps \ref{algo: polysolve: kernel step}, \ref{algo: polysolve: iwasawa step}, and \ref{eigenalgorithm changed}. Note that if $k$ is any finite precision field for which these steps can be implemented, the general recipe of \cite[Algorithm 1]{TMV2018truncated} allows us to compute the approximate solutions. Examples include ${k = \mbr, \mbc, \mbq_p, \mbf_q(\!(t)\!)}$.
	
	Finally, instead of computing a truncated normal form in steps $1$ and $2$, we may simply use a Gr\"obner basis instead. For systems where a Gr\"obner basis is already known or where the polynomials in the basis have small coefficients, the numerous optimizations of Gr\"obner basis algorithms are likely preferable from a practical standpoint. The major advantage of our implementation is the fact that the coefficient size is capped at $p^N$.

\section{Acknowledgements}

	I would like to thank Yue Ren, David Roe, and Simon Telen for helpful discussions about the article. I would also like to thank Marta Panizzut, Emre Sert\"oz, and Bernd Sturmfels for their helpful comments.


\renewcommand{\MR}[1]{}


\begin{bibdiv}
	\begin{biblist}

	\bib{TMV2018truncated}{article}{
		author={van Barel, Marc},
		author={Mourrain, Bernard},
		author={Telen, Simon},
		title={Solving polynomial systems via truncated normal forms},
		journal={SIAM J. Matrix Anal. Appl.},
		volume={39},
		date={2018},
		number={3},
		pages={1421--1447},
		issn={0895-4798},
		review={\MR{3857888}},
		doi={10.1137/17M1162433},
	}
	
	\bib{caruso2017computationswithpadics}{misc}{
		author = {Caruso, Xavier},
		title = {Computations with $p$-adic numbers},
		year = {2017},
		note = {arXiv:1701.06794},
	}
	
	\bib{CRV2015linear}{article}{
		author={Caruso, Xavier},
		author={Roe, David},
		author={Vaccon, Tristan},
		title={$p$-adic stability in linear algebra},
		conference={
			title={ISSAC'15---Proceedings of the 2015 ACM International Symposium
				on Symbolic and Algebraic Computation},
		},
		book={
			publisher={ACM, New York},
		},
		date={2015},
		pages={101--108},
		review={\MR{3388288}},
	}

	\bib{CRV2017characteristic}{article}{
		author={Caruso, Xavier},
		author={Roe, David},
		author={Vaccon, Tristan},
		title={Characteristic polynomials of $p$-adic matrices},
		conference={
			title={ISSAC'17---Proceedings of the 2017 ACM International Symposium
				on Symbolic and Algebraic Computation},
		},
		book={
			publisher={ACM, New York},
		},
		date={2017},
		pages={389--396},
		review={\MR{3703711}},
	}
	
	\bib{coxlittleoshea2015ideals}{book}{
		author={Cox, David A.},
		author={Little, John},
		author={O'Shea, Donal},
		title={Ideals, varieties, and algorithms},
		series={Undergraduate Texts in Mathematics},
		edition={4},
		note={An introduction to computational algebraic geometry and commutative
			algebra},
		publisher={Springer, Cham},
		date={2015},
		pages={xvi+646},
		isbn={978-3-319-16720-6},
		isbn={978-3-319-16721-3},
		review={\MR{3330490}},
		doi={10.1007/978-3-319-16721-3},
	}
	
	\bib{dixon1982exact}{article}{
		author={Dixon, John D.},
		title={Exact solution of linear equations using $p$-adic expansions},
		journal={Numer. Math.},
		volume={40},
		date={1982},
		number={1},
		pages={137--141},
		issn={0029-599X},
		review={\MR{681819}},
		doi={10.1007/BF01459082},
	}
	
	\bib{Fulman2002random}{article}{
		author={Fulman, Jason},
		title={Random matrix theory over finite fields},
		journal={Bull. Amer. Math. Soc. (N.S.)},
		volume={39},
		date={2002},
		number={1},
		pages={51--85},
		issn={0273-0979},
		review={\MR{1864086}},
		doi={10.1090/S0273-0979-01-00920-X},
	}
	
	\bib{nemohecke}{article}{
		author={Fieker, Claus},
		author={Hart, William},
		author={Hofmann, Tommy},
		author={Johansson, Fredrik},
		title={Nemo/Hecke: computer algebra and number theory packages for the
			Julia programming language},
		conference={
			title={ISSAC'17---Proceedings of the 2017 ACM International Symposium
				on Symbolic and Algebraic Computation},
		},
		book={
			publisher={ACM, New York},
		},
		date={2017},
		pages={157--164},
		review={\MR{3703682}},
	}
	
	\bib{Kedlaya2010differential}{book}{
		author={Kedlaya, Kiran S.},
		title={$p$-adic differential equations},
		series={Cambridge Studies in Advanced Mathematics},
		volume={125},
		publisher={Cambridge University Press, Cambridge},
		date={2010},
		pages={xviii+380},
		isbn={978-0-521-76879-5},
		review={\MR{2663480}},
		doi={10.1017/CBO9780511750922},
	}

	\bib{rutishauser1958eigenvalue}{article}{
		author={Rutishauser, Heinz},
		title={Solution of eigenvalue problems with the $LR$-transformation},
		journal={Nat. Bur. Standards Appl. Math. Ser.},
		volume={1958},
		date={1958},
		number={49},
		pages={47--81},
		review={\MR{0090118}},
	}
	

	\bib{Schikhof1984ultrametric}{book}{
		author={Schikhof, W. H.},
		title={Ultrametric calculus},
		series={Cambridge Studies in Advanced Mathematics},
		volume={4},
		note={An introduction to $p$-adic analysis},
		publisher={Cambridge University Press, Cambridge},
		date={1984},
		pages={viii+306},
		isbn={0-521-24234-7},
		review={\MR{791759}},
	}

	\bib{Storjohann2001frobenius}{article}{
		author={Storjohann, Arne},
		title={Deterministic computation of the Frobenius form (extended
			abstract)},
		conference={
			title={42nd IEEE Symposium on Foundations of Computer Science},
			address={Las Vegas, NV},
			date={2001},
		},
		book={
			publisher={IEEE Computer Soc., Los Alamitos, CA},
		},
		date={2001},
		pages={368--377},
		review={\MR{1948725}},
	}
	
	\bib{wilkinson1965convergence}{article}{
		author={Wilkinson, J. H.},
		title={Convergence of the ${\rm LR}$, ${\rm QR}$, and related algorithms},
		journal={Comput. J.},
		volume={8},
		date={1965},
		pages={77--84},
		issn={0010-4620},
		review={\MR{0183108}},
		doi={10.1093/comjnl/8.3.273},
	}

\end{biblist}
\end{bibdiv}
\end{document}